\documentclass[12pt, reqno]{amsart}
\usepackage{amsmath, amsthm, amscd, amsfonts, amssymb, graphicx, color}
\usepackage[bookmarksnumbered, colorlinks, plainpages]{hyperref}
\usepackage{tikz-cd}
\usepackage{mathrsfs} 

\calclayout

\newtheorem{theorem}{Theorem}[section]
\newtheorem{lemma}[theorem]{Lemma}
\newtheorem{proposition}[theorem]{Proposition}

\theoremstyle{definition}
\newtheorem{definition}[theorem]{Definition}

\theoremstyle{remark}
\newtheorem{remark}[theorem]{Remark}
\numberwithin{equation}{section}

\DeclareMathOperator{\di}{div}
\DeclareMathOperator{\spn}{span}
\DeclareMathOperator{\supp}{supp}
\DeclareMathOperator{\Var}{Var}

\usepackage{todonotes}
\usepackage{cancel}
\usepackage{physics}
\usepackage{xspace}

\begin{document}
\setcounter{page}{1}

\centerline{}

\centerline{}

\title[Cheeger inequality in CC-spaces]{On the Cheeger inequality in Carnot-Carath\'eodory spaces}

\author[M. Kluitenberg]{Martijn Kluitenberg$^1$}

\address{$^{1}$ Bernoulli Institute, University of Groningen, NL.}
\email{\textcolor[rgb]{0.00,0.00,0.84}{m.kluitenberg@rug.nl}}

\subjclass[2020]{Primary 58J50; Secondary 35P15, 58C40}

\keywords{CC-space, Sub-Laplacian, Cheeger inequality, Neumann boundary conditions, Nodal domains}

\date{\today}
 
\begin{abstract}
    We generalize the Cheeger inequality, a lower bound on the first nontrivial eigenvalue of a Laplacian, to the case of geometric sub-Laplacians on rank-varying Carnot-Carath\'eodory spaces and we describe a concrete method to lower bound the Cheeger constant. The proof is geometric, and works for Dirichlet, Neumann and mixed boundary conditions. One of the main technical tools in the proof is a generalization of Courant's nodal domain theorem, which is proven from scratch for Neumann and mixed boundary conditions. Carnot groups and the Baouendi-Grushin cylinder are treated as examples. 
\end{abstract} \maketitle

\tableofcontents

\section{Introduction}

In this article, we generalize the Cheeger inequality \cite{cheeger2015lower}, which is a geometric lower bound on the spectrum of the Laplacian, to possibly rank-varying Carnot-Carath\'eodory spaces (\textit{CC-spaces} or \textit{sub-Riemannian manifolds}). These geometries occur naturally when studying systems with constraints, for example in non-holonomic mechanics or thermodynamics \cite{ABB, car09}. Additionally, the motion of quantum particles in a CC-space, which can be modeled by sub-Laplacians, has applications to quantum systems \cite{boscain2002optimal}. \\

The spectral properties of sub-Laplacians have gotten quite some attention lately, see for example \cite{cdv2022spectral, boscain2013laplace, prandi2018quantum, franceschi2020essential, ES}. However, research into eigenvalue bounds for sub-Laplacians is still quite limited, despite this being an important topic in (classical) spectral geometry.  For example, the techniques of \cite{Asma14} only apply to constant-rank CC-spaces. Beyond applications to spectral geometry, the Cheeger constant of a manifold can be interesting in its own right. See \cite{franceschi2022cheeger} for a nice overview. \\

Cheeger-type inequalities for CC-spaces were already considered in the special case of Carnot groups by \cite{montefalcone2013geometric} and in the very general setting of \cite{franceschi2022cheeger}. What distinguishes this paper from \cite{franceschi2022cheeger} is our main aim and the resulting different generalizations. There, the authors were looking for the minimal assumptions to establish Cheeger inequalities. In contrast, here we look for a geometric argument that furthermore allows us to look at different boundary conditions. \\

The main result of this paper is a Cheeger inequality for domains $\Omega \subseteq M$, where $M$ is a CC-space, and where we impose Neumann boundary conditions on $\partial \Omega$. Below, $\lambda_2^N(\Omega)$ denotes the first nontrivial Neumann eigenvalue of the sub-Laplacian, which is defined below.

\begin{theorem}[Neumann-Cheeger inequality for CC-spaces]\label{thm:introNC}
	Let $M$ be a CC-space and $\Omega \subseteq M$ be a connected bounded domain with piecewise smooth boundary. Assume that
    \begin{itemize}
	\item[\textbf{(C)}] Either the topological dimension of $M$ is 2, or the manifold, the volume form $\omega$ and the vector fields $X_1 , \dots , X_m$ defining the sub-Riemannian structure are all real-analytic. 
\end{itemize} 
  Then, we have 
	\begin{equation}
	\lambda_2^N(\Omega) \geq \frac{1}{4}h_N(\Omega)^2,
	\end{equation}
 where \begin{equation}
	h_N(\Omega) = \inf_\Sigma \frac{\sigma(\Sigma)}{\min \{ \omega(\Omega_1) , \omega(\Omega_2) \}},
	\end{equation} 
	where the infimum is taken over all smooth (not necessarily connected) hypersurfaces $\Sigma \subseteq \Omega$ that separate $\Omega$ into two disjoint open sets $\Omega_1$ and $\Omega_2$. Here, $\omega$ is a smooth measure on $M$ and $\sigma$ is the induced surface measure.
\end{theorem}

The usefulness of this lower bound depends on one's ability to compute or estimate the Cheeger constant. In Section \ref{sec:mfmc}, we present a general method of obtaining lower bounds on it in terms of ``test vector fields'', which is based on the max-flow min-cut theorem of \cite{grieser2006first}. \\

The proof of Theorem \ref{thm:introNC} follows the same flavor of the Euclidean construction, but its adaptation to CC-spaces requires several results to be revisited in a more general setting. Some of these are of interest in their own right, for example, Courant's nodal domain theorem \cite{courant1923allgemeiner}, which was obtained in \cite{ES} for the Dirichlet case and which we here generalize to more general boundary conditions. For the precise statement, we refer to Section \ref{sec:courant}. \\

A CC-space can be understood as a manifold $M$ together with a \emph{bracket-generating} family of vector fields $X_1 , \dots , X_m$, i.e. any vector in the tangent bundle is eventually obtained by taking iterated Lie brackets of the $X_j$'s. We give more precise definitions in Section \ref{sec:CC}. The sub-Laplacian, like in the Euclidean case, is the ``divergence of gradient''. Indeed, the \emph{horizontal gradient} of a function can be defined as 
\begin{equation}
    \nabla_H u = \sum_{i = 1}^m (X_i u) X_i
\end{equation}
and the \emph{divergence} is taken with respect to some fixed volume form $\omega$ on $M$: 
\begin{equation}
    L_X \omega = \di_\omega (X) \; \omega.
\end{equation}
In the above equation, $L_X$ denotes the Lie derivative on forms. The \emph{(geometric) sub-Laplacian} is then $\Delta u := \di_\omega (\nabla_H u)$.  \\ 

Important examples of CC-spaces are the Heisenberg group and the Grushin plane. The Heisenberg group is an example of an equiregular CC-space, as are all Carnot groups. The Grushin plane, by contrast, is \emph{rank-varying}, and has a singular set $Z$ outside of which the structure is Riemannian. In Section \ref{sec:examples}, we study Carnot groups and the Grushin cylinder (where one direction of the Grushin plane is compactified to a circle) in detail. In particular, we compute the spectrum  of $(0,1) \times \mathbb{S}^1$ equipped with the Grushin structure and a regular volume form. \\

For the proof of Courant's theorem, we discuss the analytical assumptions \textbf{(S)} and \textbf{(C)} on the CC-structure. Assumption \textbf{(C)} was present in Theorem \ref{thm:introNC}, while 
\begin{itemize}
	\item[\textbf{(S)}] Either $\Omega = M$ or the boundary of $\Omega$ is smooth and contains no characteristic points, i.e. no points such that $$T_p(\partial \Omega) \subseteq \mathcal{D}_p = \spn\{ X_1(p),\dots,X_m(p) \} .$$
\end{itemize}
 Under these two assumptions, any eigenfunction $u$ of the sub-Laplacian corresponding to $\lambda = \lambda_k$ has at most $k$ nodal domains. This was shown in \cite{ES} for Dirichlet boundary conditions. We extend the proof in \cite{ES} to mixed and Neumann boundary conditions. We also extend a recent result of \cite{frank2024courant} to mixed boundary conditions, which allows us to bypass assumption \textbf{(S)}.  \\

In the Euclidean case, the Cheeger inequality is sharp \cite[Section 5.2.2]{TSG}. Moreover, in the case of closed Riemannian manifolds, we have the following optimality result due to Buser \cite{buser1978ungleichung}: For all $h > 0,$ $\varepsilon > 0$ and $k \in \mathbb{N}$, there is a closed 2-dimensional Riemannian manifold which has Cheeger constant $h$ and $\lambda_{k + 1}(M) < \frac{h^2}{4} + \varepsilon$. We expect sharpness in the CC-case to be a delicate and complex problem, which we decided to postpone to future research.
In Section \ref{sec:examples}, however, we look at some model examples to test our results in context. \\

Throughout the paper, we assume that all CC-spaces come equipped with a fixed smooth volume form $\omega$. In general, there is no canonical way to choose $\omega$. For Carnot groups, the natural choice is the left Haar measure. For the Grushin plane $M$, the ``natural'' volume form is then the Riemannian volume, which is defined on $M \setminus Z$.  However, this volume form blows up on the singular set, hence our results do not apply to it, though this case is studied, for example in \cite{prandi2018quantum,franceschi2020essential}. We take the agnostic point of view as to which volume form is the ``correct'' one, and just fix any smooth one. \\

The article is structured as follows: We start in Section \ref{sec:preliminaries} by defining CC-spaces, generating families and their main properties, and we define the sub-Laplacian. We then study its self-adjoint extensions and the corresponding boundary conditions. Afterwards, we recall the notion of horizontal perimeter and the coarea formula, which is an important technical tool. In Section \ref{sec:courant}, we prove Courant's theorem for Neumann and mixed boundary conditions. Then in Sections \ref{sec:DC} and \ref{sec:NC}, we state and prove Cheeger's inequality for Dirichlet and Neumann boundary conditions respectively. In Section \ref{sec:mfmc}, we present a method of obtaining lower bounds on the Cheeger constant. Finally, in Section \ref{sec:examples}, we consider how the Cheeger inequality applies to Carnot groups and to the Gruhsin cylinder. 

\section{Preliminaries}\label{sec:preliminaries}

\subsection{Carnot-Carath\'eodory spaces}\label{sec:CC}

We start by giving a very general definition of a CC-space. After that, we introduce the concept of a \emph{generating family}, which should make things more transparent.

\begin{definition}
    Let $M$ be either $\mathbb{R}^n$ or an $n$-dimensional smooth, connected, compact manifold without boundary. A \emph{horizontal structure} on $M$ is a pair $(\vb{U}, f)$ where $\vb{U} \xrightarrow{\pi_{\vb{U}}} M$ is a Euclidean vector bundle over $M$, and $f : \vb{U} \to TM$ is a morphism of vector bundles. The set of \emph{horizontal vector fields} is defined to be $$\mathfrak{X}_H(M) = \{ f \circ \sigma : \sigma \in \Gamma(\vb{U}) \} \subseteq \mathfrak{X}(M).$$
    Let $\mathrm{Lie}(\mathfrak{X}_H(M))$ be the smallest Lie subalgebra of $\mathfrak{X}(M)$ containing $\mathfrak{X}_H(M)$ and set $\mathcal{D}_p := \{ X(p) : X \in \mathrm{Lie}(\mathfrak{X}_H(M)) \} \subseteq T_p M$. We say that $(\vb{U},f)$ is \emph{bracket-generating} if $\mathcal{D}_p = T_p M$ for all $p \in M$. A \emph{CC-space} is a quadruple $(M,\vb{U},f,\omega)$ such that $(\vb{U},f)$ is a bracket-generating horizontal structure on $M$ and $\omega \in \Omega^n(M)$ is a smooth volume form. 
\end{definition}

Let $M$ be a CC-space and $v \in \mathcal{D}_p$. Define the norm
\[
    \|v\| = \min\{ |\vb{u}| : \vb{u} \in \vb{U}_p \text{ and } f(\vb{u}) = v \},
\]
where $|\cdot|$ denotes the norm on the fiber $\vb{U}_p$. It turns out \cite[Exercise 3.9]{ABB} that the norm $\| \cdot \|$ is induced by an inner product $g_p(\cdot , \cdot)$ on $\mathcal{D}_p$, which can be recovered from the norm by polarization:
\begin{equation}
    g_p(X,Y) = \frac{\|X + Y\|^2 - \|X - Y \|^2}{4}.
\end{equation}
The \emph{CC-metric} $g$ is defined pointwise: If $X, Y \in \mathfrak{X}_H(M)$, we set
\[
    g(X,Y) (p) = g_p (X(p), Y(p)). 
\]
It is often assumed that the rank $r(p) = \dim(\mathcal{D}_p)$ is constant across $M$. In this case, it is possible to find a local orthonormal frame $Y_1 , \dots , Y_k$ for the distribution. Although this simplifies life a lot, we do not make this assumption. 

\begin{definition}
    A curve $\gamma : I \to M$ is called \emph{horizontal} if there exists a measurable and essentially bounded function $\vb{u} : I \to \vb{U}$ such that $\vb{u}(t) \in \vb{U}_{\gamma(t)}$ and $\gamma'(t) = f(\vb{u}(t))$ for almost every $t \in I$. The \emph{length} of a horizontal curve is 
    \[
    \ell(\gamma) := \int_I \sqrt{g_{\gamma(t)}(\gamma'(t), \gamma'(t))} \dd{t}.
    \]
    Finally, we define the \emph{Carnot-Carath\'eodory distance} between two points $p,q \in M$ as 
    \[
    d_{CC}(p,q) := \inf \{ \ell(\gamma) : \gamma \text{ is a horizontal curve from $p$ to $q$} \}.
    \]
\end{definition}
If $M$ is a CC-space, it is a standard result \cite[Theorem 3.31]{ABB} that any two points of $M$ can be connected by a horizontal curve. Moreover, $d_{CC}$ is a distance function on $M$ whose induced topology is equivalent to the manifold topology. \\

Without loss of generality, one may assume that the bundle $\vb{U}$ is trivial, i.e. $\vb{U} \simeq M \times \mathbb{R}^m$ \cite[Section 3.1.4]{ABB}. Let $\{\vb{e}_1 , \dots , \vb{e}_m \}$ be a global orthonormal frame for $\vb{U}$, i.e. $\vb{e}_i : M \to \vb{U}$ are sections of $\vb{U}$ such that 
\[
\langle \vb{e}_i(p), \vb{e}_j(p) \rangle = \delta_{ij},
\]
where $\langle \cdot , \cdot \rangle$ denotes the inner product on the fiber $\vb{U}_p$. The vector fields $X_i := f \circ \vb{e}_i$ are called a \emph{generating family} for the horizontal structure on $M$. Many concepts in CC-geometry become clearer if viewed through the lens of a generating family. For example, any section $\sigma \in \Gamma (\vb{U})$ can be written as $\sigma(p) = \sum_{i = 1}^m u_i(p)  \vb{e}_i (p)$ for smooth functions $u_i : M \to \mathbb{R}$. Since the map $f : \vb{U} \to TM$ is linear on fibers, we get 
\[
f \circ \sigma (p) = \sum_{i = 1}^m u_i(p) X_i (p).
\]
That is, any horizontal vector field can be written as a $C^\infty (M)$-linear combination of the generating family. Similarly, a curve $\gamma : I \to M$ is horizontal if there are functions $u_i \in L^\infty (I)$ such that 
\[
\gamma'(t) = \sum_{i = 1}^m u_i (t) X_i(\gamma (t))
\]
for almost every $t \in I$. \\

The generating family $X_1 , \dots , X_m$ is not necessarily orthonormal with respect to the sub-Riemannian metric $g$. We can only guarantee this in the simple case where $f$ is injective on fibers. We do however preserve a ``completeness''-type relation. We omit the proof, which comes down to a linear algebra calculation at every point.  

\begin{proposition}
    Let $M$ be a CC-space with generating family $X_1 , \dots , X_m$. Then, 
     \begin{equation}\label{eq:h-grad-check-2}
         \sum_{i=1}^m g(X_i,X) X_i  = X.
     \end{equation}
     for all $X \in \mathfrak{X}_H(M)$. In particular, the expression $\sum_{i=1}^m g(X_i,X) X_i$ is independent of the generating family. 
\end{proposition}

\subsection{Sub-Laplacians}

We now turn to spectral geometry. In this subsection, we shall define the geometric sub-Laplacian, which is a generalization of the Laplacian on Euclidean space $\mathbb{R}^n$ in the sense that both can be written as ``divergence of gradient''. Let $(M,\mathcal{D},g, \omega)$ be a CC-space. We define the \emph{horizontal gradient} of a function $u \in C^\infty (M)$ as the unique horizontal vector field such that
\begin{equation}\label{eq:h-grad-defn}
	g( \nabla_H u, X ) = Xu, \qquad \forall X \in \mathfrak{X}_H(M).
\end{equation}
The horizontal gradient has a simpler expression in terms of a generating family \cite[Exercise 21.1]{ABB}. The proof follows directly from (\ref{eq:h-grad-check-2}). 
\begin{proposition}
    If $X_1 , \dots , X_m$ is a generating family, then
\begin{equation}\label{eq:h-grad-frame}
    \nabla_H u = \sum_{i = 1}^m ( X_i u) X_i.
\end{equation}
In particular, the expression $\sum_{i = 1}^m ( X_i u) X_i$ is independent of the generating family.
\end{proposition}

We define the \emph{divergence} of a vector field with respect to the fixed volume form $\omega$ by the equation
\begin{equation}
	L_X \omega = \di_\omega (X) \cdot \omega, \qquad X \in \mathfrak{X}(M),
\end{equation}
where $L_X$ denotes the Lie derivative. Finally, we define the \emph{(geometric) sub-Laplacian} as the divergence of the horizontal gradient: 
\begin{equation}
    \Delta u = \di_\omega (\nabla_H u), \quad u \in C^\infty (M). 
\end{equation}
In terms of a generating family, we have \cite[p. 577]{ABB}
\begin{equation}
    \Delta u = \sum_{i = 1}^m \left( X_i^2 u + \di_\omega(X_i) X_i u \right),
\end{equation}
i.e. the sub-Laplacian is the sum of squares $\sum_{i = 1}^m X_i^2$ plus a first-order term.

\begin{remark}
    The geometric sub-Laplacian belongs to a class of second-order differential operators on $M$. If $(X_0,X_1,\dots , X_m)$ is any family of bracket-generating vector fields on $M$, then the differential operator $L = \sum_{i = 1}^m X_i^2 + X_0$ is called a sub-Laplacian \cite{gordina2016sub} on $M$. By H\"ormander's theorem \cite{hormander1967hypoelliptic}, any such operator is hypoelliptic. In \cite{ES}, the operator $\sum_{i = 1}^m X_i ^* X_i$ is considered, where $X_i^* = - X_i - \di_\omega (X_i)$ is the formal adjoint of $X_i$. This operator is equal to the geometric sub-Laplacian if $(X_1 , \dots , X_m)$ is an orthonormal frame for the distribution. 
\end{remark}

Let $\Omega \subseteq M$ be a bounded domain with a piecewise smooth boundary. We have the following simple consequence of the divergence theorem (see \cite[p. 13]{Asma14}): 

\begin{proposition}[Sub-Riemannian Gauss-Green formula]
	The identity
	\begin{equation}\label{eq:GG}
		\int_\Omega (\Delta u) v \; \omega + \int_\Omega g(\nabla_H u, \nabla_H v )\; \omega = \oint_{\partial \Omega} v \; \iota_{\nabla_H u}\;  \omega, 
	\end{equation}
    holds for all  $u,v \in C^\infty (\overline{\Omega})$.  
\end{proposition} 

Let $L^2(\Omega)$ be the space of functions on $\Omega$ which are square-integrable with respect to $\omega$. The symbols $\| \cdot \|$ and $\langle \cdot , \cdot \rangle$ without any subscript denote the $L^2$-norm and inner product respectively. We now construct the appropriate class of Sobolev spaces to study sub-Laplacians on $\Omega$. Define 
\[
S^1 (\Omega) := \{ u \in L^2(\Omega) : \nabla_H u \in L^2(\Omega) \}.
\]
Here, we understand that $\nabla_H u = \sum_{i = 1}^m (X_i u) X_i$ is defined in the sense of distributions. We equip $S^1 (\Omega)$ with the norm
\begin{equation}
	\|u\|_{S^1}^2 := \|u\|^2 + \| \nabla_H u\|^2. 
\end{equation}
We further define $S_0^1(\Omega)$ as the closure of $C^\infty_0(\Omega)$ taken with respect to the $S^1$-norm. As in the Euclidean case, functions in $S_0^1(\Omega)$ effectively satisfy Dirichlet boundary conditions. We will also need to discuss mixed boundary conditions, where we enforce Dirichlet boundary conditions on a subset $\Gamma \subseteq \partial \Omega$ and Neumann boundary conditions on $\partial \Omega \setminus \Gamma$. The set $\Gamma$ is assumed to have finitely many connected components. Mixed boundary conditions can be encoded in the space $C^\infty_{0,\Gamma}(\Omega)$ of smooth functions whose support doesn't intersect $\overline{\Gamma}$: 
\[
C_{0,\Gamma}^\infty (\Omega) := \{ u \in C^\infty (\Omega)  : \supp(u) \cap \overline{\Gamma} = \varnothing \}.
\]
We then define the corresponding space $S^1_{0,\Gamma}(\Omega)$ as the closure of $C_{0,\Gamma}^\infty (\Omega)$ taken with respect to the $S^1$-norm. 

\begin{proposition}
    Let $M$ be a CC-space and $\Omega \subseteq M$ open. Then, $S^1(\Omega), S^1_0(\Omega)$ and $S^1_{0,\Gamma}(\Omega)$ are Hilbert spaces. 
\end{proposition}

\begin{proof}
    The fact that $S^1(\Omega)$ is a Hilbert space is a standard fact, see e.g. \cite[p. 1083]{garofalo1996isoperimetric}. Further, $S^1_0(\Omega)$ and $S^1_{0,\Gamma}(\Omega)$ are closed subspaces of $S^1(\Omega)$, so they are also complete. 
\end{proof}

The next Lemma will play an important role in the proof of Courant's theorem (Section \ref{sec:courant}): 

\begin{lemma}[Vanishing Lemma]\label{lem:vanishing}
	Let $u \in S^1(\Omega)$. Suppose that for all $y \in \Gamma$ and for all $\varepsilon > 0$ there exists a neighborhood $U$ of $y$ such that $|u(x)|< \varepsilon$ for all $x \in U \cap \Omega$. Then, $u \in S^1_{0,\Gamma}(\Omega)$. 
\end{lemma}
When $\Gamma = \partial \Omega$, this becomes a Lemma about $S^1_0(\Omega)$, already proven in \cite{frank2024courant}. In fact, the method of proof in \cite{frank2024courant} carries over to the case of mixed boundary conditions. Without loss of generality, let $u \geq 0$. The function $u_\varepsilon = \min(0,u - \varepsilon)$ approximates $u$ in $S^1$-norm, and it lies in $S^1_{0,\Gamma}(\Omega)$ for all $\varepsilon > 0$. Note that $u_\varepsilon \in S^1(\Omega)$ and $\supp(u_\varepsilon) \cap \overline{\Gamma} = \varnothing$, so smooth functions approximating $u_\varepsilon$ can be taken to lie in $C^\infty_{0,\Gamma}(\Omega)$. This shows that $u_\varepsilon \in S^1_{0,\Gamma}(\Omega)$, completing the proof. For more details, see \cite[p. 12]{frank2024courant}.

\subsection{Quadratic form and extension}

 We use the notation $S^1_\bullet(\Omega)$ to denote either $S^1(\Omega), S^1_0(\Omega)$ or $S^1_{0,\Gamma}(\Omega)$. These spaces correspond to Neumann, Dirichlet and mixed boundary conditions respectively. For each boundary condition, we define a self-adjoint extension of the sub-Laplacian (defined initially on $C^\infty (\Omega)$) using quadratic forms. 
 Define 
 \begin{equation}
 	q^\bullet_{\Omega}(u,v) := \langle \nabla_H u, \nabla_H v \rangle_{L^2} = \int_\Omega g(\nabla_H u, \nabla_H v) \; \omega \qquad \forall \; u,v \in S^1_\bullet (\Omega).
 \end{equation}
 Further, the shorthand $q^\bullet_\Omega(u)$ means $q^\bullet_\Omega(u,u)$. For specific instances of the quadratic form, we use $q^D_\Omega, q^N_\Omega$, and $q^{Z,\Gamma}_\Omega$ respectively. Using this quadratic form, we define self-adjoint extensions of the sub-Laplacian, which are called Friedrichs extensions \cite[Section 3.4.3]{Bor} in the literature. Let 
 \[
 D(-\Delta_\bullet) = \{ u \in S^1_\bullet (\Omega) : v \mapsto q_\Omega (v,u) \text{ extends to a bounded linear map on $L^2(\Omega)$} \}. 
 \]
 The corresponding operator is defined by the Riesz representation theorem \cite[Theorem 2.28]{Bor}. If $u \in D(- \Delta_\bullet)$, 
 \[
 \exists ! \; w \in L^2(\Omega) : q_\Omega(v,u) = \langle v,w \rangle \qquad \forall v \in S^1_\bullet (\Omega)
 \]
 and we define $- \Delta_\bullet u := w$. In case of specific boundary conditions, we use the notation $\Delta_D, \Delta_N$ or $\Delta_{Z,\Gamma}.$ \\

The following well-known results from spectral theory are recalled here for later convenience: 

\begin{theorem}[Weak spectral theorem]\label{thm:weak-spectral}
	Let $\Omega \subseteq M$ be a bounded domain with a piecewise smooth boundary. The operator $- \Delta_\bullet : D(- \Delta_\bullet) \to L^2(\Omega)$ is a self-adjoint extension of the sub-Laplacian. The spectrum is discrete,
	\begin{equation}
		0 \leq \lambda_1^\bullet(\Omega) \leq \lambda_2^\bullet(\Omega) \leq \dots \uparrow \infty,
	\end{equation}
	 counted with multiplicity. Moreover, $L^2(\Omega)$ admits an orthonormal basis of eigenfunctions $\{ u_n \in L^2(\Omega) : n \in \mathbb{N} \}$ of $- \Delta_\bullet$.
\end{theorem} 

\begin{theorem}[Min-max principle]\label{thm:min-max}
	Let $\Omega \subseteq M$ be a bounded domain with a piecewise smooth boundary. A function $u \in S^1_\bullet (\Omega) \setminus \{ 0 \}$ is an eigenfunction corresponding to $\lambda_1^\bullet (\Omega)$ if and only if $u$ minimizes the Rayleigh quotient, i.e.
	\begin{equation}
		R^\bullet_\Omega [u] := \frac{q^\bullet_{\Omega}(u)}{\|u\|^2} = \frac{\|\nabla_H u\|^2}{\|u\|^2}
	\end{equation}
	over the form domain $S^1_\bullet (\Omega)$. In this case, $R^\bullet_\Omega[u] = \lambda_1^\bullet (\Omega)$. For the higher eigenvalues, let $L_{k - 1} := \spn \{ u_1 , \dots , u_{k - 1} \}$. A function $u \in S^1_\bullet (\Omega) \setminus \{ 0 \}$ is an eigenfunction corresponding to $\lambda_k^\bullet(\Omega)$ if and only if $u$ minimizes the Rayleigh quotient over $(S^1_\bullet (\Omega) \setminus \{ 0 \}) \cap L_{k - 1}^\perp$, and in this case $\lambda_k^\bullet (\Omega) = R^\bullet_\Omega[u].$
\end{theorem}

The proof proceeds as in \cite{ES}. Dirichlet or Neumann eigenvalues are denoted by $\lambda^D_k(\Omega)$ and $\lambda^N_k(\Omega)$ respectively. For mixed boundary conditions with Dirichlet conditions on $\Gamma$, the eigenvalues are denoted $\lambda_k^Z(\Omega, \Gamma)$.

\subsection{Boundary conditions}

The min-max principle above makes no reference to boundary conditions. However, the choice of domain for the quadratic form, and hence for the operator, forces a particular choice of boundary conditions. Choosing $S^1_0(\Omega)$ enforces Dirichlet boundary conditions: eigenfunctions of this problem vanish on $\partial \Omega$. The case of Neumann boundary conditions is a bit more subtle.
In the Euclidean case, using $H^1(\Omega)$ as the form domain enforces Neumann boundary conditions, i.e. $\partial_{\vb{n}} u = 0$ on $\partial \Omega$, where $\vb{n}$ denotes the outward pointing normal vector field to $\partial \Omega$. In this subsection, we see how this is generalized to sub-Laplacians. We start by introducing the correct notion of a normal vector.

\begin{definition}[Horizontal normal]
    Let $M$ be a CC-space and $\Sigma \subseteq M$ a smooth hypersurface. We say that $p \in \Sigma$ is a characteristic point of $\Sigma$ if $\mathcal{D}_p \subseteq T_p \Sigma$. If $p \in \Sigma$ is non-characteristic, we define a \emph{horizontal normal} to $\Sigma$ as a unit vector $\vb{n}_H(p)$ in $\mathcal{D}_p$ orthogonal to $\mathcal{D}_p \cap T_p \Sigma$. If $p \in \Sigma$ is characteristic, we set $\vb{n}_H(p) = 0$. 
\end{definition}

\begin{remark}
    In the above definition, we have restricted everything to the space $\mathcal{D}_p$, on which we have a well-defined inner product $g_p$. We have hence defined $\vb{n}_H(p)$ uniquely up to a sign. Indeed, for a non-characteristic point, $\mathcal{D}_p + T_p \Sigma = T_p M$, hence $\mathcal{D}_p \cap T_p \Sigma \subseteq \mathcal{D}_p$ has codimension 1.
\end{remark}

We now turn to the generalization of Neumann boundary conditions: 

\begin{proposition}\label{prop:boundary}
	Let $\Omega \subseteq M$ be a bounded domain with piecewise smooth boundary. Assume that $u \in D(- \Delta_N) \cap C^\infty (\overline{\Omega}).$ Then, horizontal Neumann boundary conditions hold on $\partial \Omega$, i.e. 
    \begin{equation}\label{eq:h-Neumann}
        g(\nabla_H u , \vb{n}_H) = 0, \qquad \text{ on $\partial \Omega$}.
    \end{equation}
\end{proposition}

\begin{proof}

From (\ref{eq:GG}), it is clear that $\iota_{\nabla_H u} \omega = 0$ as an $(n - 1)$-form on $\partial \Omega$. It remains to prove (\ref{eq:h-Neumann}) from this. Fix $p \in \partial \Omega$. If $p$ is characteristic, then $\vb{n}_H(p) = 0$, so we may assume $p$ is non-characteristic. Write $\nabla_H u(p) = a \vb{n}_H(p) + \vb{v}$, where $\vb{v} \in \mathcal{D}_p \cap T_p (\partial \Omega)$ and choose a basis $\vb{v}_1 , \dots , \vb{v}_{n - 1}$ for $T_p(\partial \Omega)$. Note that 
\begin{align*}
0 &= (\iota_{\nabla_H u} \omega)_p (\vb{v}_1 , \dots , \vb{v}_{n - 1}) = \omega_p(\nabla_H u(p), \vb{v}_1 , \dots , \vb{v}_{n - 1}) \\
&= a \; \omega_p (\vb{n}_H(p),\vb{v}_1 , \dots , \vb{v}_{n - 1}) + \omega_p (\vb{v},\vb{v}_1 , \dots , \vb{v}_{n - 1}),
\end{align*}
where the last term is zero because the set $\{ \vb{v},\vb{v}_1 , \dots , \vb{v}_{n - 1} \}$ is linearly dependent. The proof is completed by noting that $a =  g(\nabla_H u , \vb{n}_H) = 0$.
\end{proof}

\subsection{Coarea formula}

Let $\Omega \subseteq M$ be an open subset of a CC-space with a piecewise smooth boundary, and let $(X_1 , \dots , X_m)$ be a generating family. In this subsection, we state our main technical tool, which is a version of the coarea formula. This formula involves a decomposition of $\Omega$ into level sets $\{ x \in \Omega : u(x) = t \}$ where $u$ is a smooth function. In order to define a measure on these hypersurfaces that is compatible with the CC-structure, we first discuss the notion of \emph{horizontal perimeter}. This discussion is based on \cite{garofalo1996isoperimetric}. \\

Define
\[
\mathscr{F} (\Omega) := \left\{ \varphi = (\varphi_1 , \dots , \varphi_m) \in C^\infty_0(\Omega , \mathbb{R}^m) : \|\varphi\|_\infty \leq 1 \right\},
\]
where 
\[
\|\varphi\|_\infty = \sup_{x \in \Omega} \left( \sum_{i = 1}^m |\varphi_j(x)|^2 \right)^{1/2}.
\]
For $u \in L^1(\Omega)$, define the \emph{horizontal variation}
\begin{equation}\label{eq:h-variation}
\Var_H(u;\Omega) := \sup_{\varphi \in \mathscr{F}(\Omega)} \int_\Omega u(x) \sum_{j = 1}^m X_j ^* \varphi_j(x) \; \omega,
\end{equation}
where $X_j^* = - X_j - \di_\omega (X_j)$ is the \emph{formal adjoint} of $X_j$ \cite[Section 10.1]{Nic}.
The space of functions of bounded horizontal variation $$BV_H(\Omega) := \{ u \in L^1(\Omega) : \Var_H(u;\Omega) < \infty \}$$ is a Banach space under the norm $\|u\|_{BV} = \|u\|_{L^1} + \Var_H (u;\Omega)$ \cite{garofalo1996isoperimetric}. \\

Suppose that $u$ is a smooth function. The first order of business is to rewrite the horizontal variation in terms of the horizontal gradient. 

\begin{proposition}\label{propn:h-var-h-grad}
    Let $\Omega \subseteq M$ be an open subset of a CC-space with a piecewise smooth boundary and let $u \in C^\infty (\Omega)$. Then, 
    \begin{equation}\label{eq:h-div}
        \int_\Omega u(x) \sum_{j = 1}^m X_j ^* \varphi_j(x) \; \omega = \int_{\Omega} \sum_{j = 1}^m \varphi_j X_j u \; \omega
    \end{equation}
    for all $\varphi = (\varphi_1 , \dots , \varphi_m) \in \mathscr{F}(\Omega)$. Hence,
    \begin{equation}\label{eq:h-div-2}
        \Var_H(u;\Omega) = \sup_{X} \left\{ \int_{\Omega} X u \; \omega \right\} 
    \end{equation}
    where the supremum is taken over all horizontal vector fields $X = \sum_{i = 1}^m \varphi_i X_i$ with $\varphi = (\varphi_1 , \dots , \varphi_m) \in \mathscr{F}(\Omega)$. The supremum (\ref{eq:h-div-2}) is attained when $X$ is parallel to $\nabla_H u$, and thus 
    \begin{equation}
        \Var_H(u;\Omega) = \int_\Omega |\nabla_H u| \; \omega.
    \end{equation}
\end{proposition}

\begin{proof}
    Equation (\ref{eq:h-div}) is a consequence of the divergence theorem, where the boundary term vanishes because $\varphi \in C^\infty_0(\Omega; \mathbb{R}^m)$. Equation (\ref{eq:h-div-2}) is clear from (\ref{eq:h-div}). For the last statement, fix a point $p \in \Omega$. For any horizontal vector field $X$, we have $|Xu (p)| = |g_p(\nabla_H u , X)| = |\nabla_H u| |X| \cos \vartheta$. Thus, to maximize $|Xu(p)|$ for horizontal vector fields of fixed norm, one needs to choose $X$ parallel to $\nabla_H u$ and of maximal length. Hence, 
    \[
    \Var_H(u;\Omega) = \int_{\Omega} \frac{1}{|\nabla_H u|} (\nabla_H u) (u) \; \omega = \int_\Omega  \frac{1}{|\nabla_H u|} g(\nabla_H u, \nabla_H u) \; \omega = \int_\Omega |\nabla_H u| \; \omega.
    \]
    This completes the proof. 
\end{proof}

If $E \subseteq M$ is measurable, then we define the \emph{horizontal perimeter} of $E$ relative to $\Omega$ by 
\[
P_H (E;\Omega) := \Var_H (\chi_E ; \Omega),
\]
where $\chi_E$ is the characteristic function of $E$. Note that the horizontal perimeter of a set is not necessarily finite. Assume that $E \subseteq M$ is an open set with a piecewise smooth boundary. A similar computation as in (\ref{eq:h-div}) shows that 
\begin{equation}\label{eq:hperim}
    P_H(E; \Omega) = \sup_{\varphi \in \mathscr{F}(\Omega)} \oint_{\partial E} \sum_{i = 1}^m \varphi_i \; \iota_{X_i}\omega.
\end{equation}
This is just the boundary term coming from the divergence theorem applied to $E$. We hence obtain the following Lemma:

\begin{lemma}
    The supremum in (\ref{eq:hperim}) is attained when $\varphi|_{\partial E}$ is a horizontal normal $\vb{n}_H$ to $\partial E$ making the integral in (\ref{eq:hperim}) positive. 
\end{lemma}

\begin{proof}
    First, suppose that $p \in \partial E$ is a characteristic point, i.e. $\mathcal{D}_p \subseteq T_p (\partial E)$. In this case, if $v_1 , \dots , v_{n - 1} \in T_p(\partial E)$, the set $\{ \varphi_p , v_1 , \dots , v_{n - 1} \}$ is linearly dependent since $\varphi_p \in \mathcal{D}_p$. Thus, $(\iota_\varphi \omega)_p = 0$ for all $\varphi \in \mathscr{F}(\Omega)$ and we can choose any value we like for $\varphi_p$ at a characteristic point. \\

    Now suppose that $p \in \partial E$ is non-characteristic. We can decompose uniquely $\varphi_p = a \vb{n}_H(p) + \vb{v}$ where $\vb{v} \in \mathcal{D}_p \cap T_p(\partial E)$. Note that $\iota_{\vb{v}} \omega$ vanishes on $\partial \Omega$, so to maximize $\iota_\varphi \omega$ over all horizontal vectors $\varphi$ with $\|\varphi\| \leq 1$, we should choose $a = \pm 1$ depending on orientation. 
\end{proof}

We have hence proven the identity 
\begin{equation}\label{eq:surfacemeasure-perimeter}
    P_H(E; \Omega) = \oint_{\partial E} \iota_{\vb{n}_H} \omega,
\end{equation}
where the sign of $\vb{n}_H$ has been chosen to make the integral (\ref{eq:surfacemeasure-perimeter}) positive. We now state the coarea formula: 

\begin{theorem}[Coarea formula]
    For any $u \in BV_H(\Omega)$, we have 
    \begin{equation}\label{eq:co-area-var}
        \Var_H(u;\Omega) = \int_\mathbb{R} P_H (\{ x \in \Omega : u(x) > t \} ; \Omega) \dd{t}. 
    \end{equation}
    If further $\Omega \subseteq M$ is an open set with a piecewise smooth boundary and $u \in C^\infty (\Omega)$, then
    \begin{equation}\label{eq:coarea}
     \int_{\Omega} |\nabla_H u| \; \omega = \int_\mathbb{R} P_H (\{ x \in \Omega : u(x) > t \} ; \Omega) \dd{t}.
\end{equation}
\end{theorem}

\begin{proof}
    Equation (\ref{eq:co-area-var}) is a result of \cite[equation (1.21)]{garofalo1996isoperimetric}, and (\ref{eq:coarea}) follows from (\ref{eq:co-area-var}) and Proposition \ref{propn:h-var-h-grad}.
\end{proof}

\section{Courant's nodal domain theorem}\label{sec:courant}

In order to prove generalizations of the Cheeger inequality, we need to make use of Courant's nodal domain theorem. Typically, the proofs are only written down in the case of Dirichlet boundary conditions. Since Courant's theorem is interesting in its own right, we decided to give full proofs here. The structure of the proof draws from \cite{ES}, but it has been extended to the case of Neumann or mixed boundary conditions. \\

Let $M$ be a CC-space and let $\Omega \subseteq M$ be a bounded connected domain with a piecewise smooth boundary. Let $u$ be an eigenfunction of the spectral problem $- \Delta_\bullet u = \lambda u$ on $\Omega$ corresponding to $\lambda = \lambda_k$. It does not matter if we consider Dirichlet, Neumann or mixed boundary conditions. Denote by $Z_u$ the nodal set of $u$, i.e. 
\[
Z_u := \{ p \in \Omega : u(p) = 0 \}.
\]
The connected components of $\Omega \setminus Z_u$ are referred to as \emph{nodal domains} of $u$. Courant's theorem states that $u$ has at most $k$ nodal domains. Because of the analytical assumptions \textbf{(S)} and \textbf{(C)}, we have 
\begin{theorem}[cf. \cite{ES} Section 1.3]
	If \textbf{(S)} is satisfied, then eigenfunctions of $- \Delta_\bullet$ are smooth up to the boundary $\partial \Omega$. If \textbf{(C)} is satisfied, then the unique continuation property holds, i.e. any solution to $(\Delta_\bullet - \lambda) u = 0$ that vanishes on a nonempty open subset of $\Omega$ vanishes identically in $\Omega$.
\end{theorem}

\begin{remark}
    For Carnot groups $G$, there is a diffeomorphism $F : \mathbb{R}^n \to G$ by some system of graded exponential coordinates. Thus, a bounded domain $\Omega \subseteq G$ can never be all of $M$, and assumption \textbf{(S)} reduces to the boundary of $\Omega$ being non-characteristic. Moreover, in the system of coordinates given by $F$, the vector fields $X_1, \dots , X_m$ have polynomial coefficients \cite[Proposition 2.4]{cassano2016some}. Thus, assumption \textbf{(C)} is superfluous for Carnot groups. 
\end{remark}


\begin{theorem}[Courant's nodal domain theorem] 
	Let $M$ be a CC-space. Let $\Omega \subseteq M$ be an open connected domain with piecewise smooth boundary, and suppose that assumptions \textbf{(S)} and \textbf{(C)} hold. 
	Let $u$ be an eigenfunction of the spectral problem $- \Delta_\bullet u = \lambda u$ with Dirichlet, Neumann, or mixed boundary conditions corresponding to $\lambda = \lambda_k^\bullet (\Omega)$. Then, $u$ has at most $k$ nodal domains. 
\end{theorem}

\begin{remark}
	As we will see in the second proof, assumption \textbf{(S)} is superfluous. 
\end{remark}

\subsection{First proof}

We now give two proofs of Courant's theorem. The first proof follows closely \cite{ES}. We start by constructing cutoff functions which localize close to $\Gamma$. For the next Lemma, and hence for the rest of the first proof, we restrict to the case where $\Gamma \subseteq \partial \Omega$ is a submanifold. This avoids the technicalities of analyzing the properties function $x \mapsto d(x,\Gamma)$ is degenerate cases. Since the second proof does not make this assumption, there is no loss of generality.  

\begin{lemma}\label{lem:cutoff_existence}
	Let $M$ be either $\mathbb{R}^n$ or a smooth closed  $n$-dimensional manifold, and let $\Omega \subseteq M$ be a bounded open set with piecewise smooth boundary. Let $\Gamma \subseteq \partial \Omega$ be a submanifold. There exists a Riemannian metric $\bar{g}$ on $M$, and functions $\chi_n \in C_0^\infty (\overline{\Omega})$ such that 
	\begin{enumerate}
		\item $\chi_n(x) = 1$ for $d_{\bar{g}}(x,\Gamma) \geq \frac{1}{n}$;
		\item $\chi_n(x) = 0$ for $d_{\bar{g}}(x,\Gamma) \leq \frac{1}{2n}$; 
		\item $\| \bar{\nabla} \chi_n\|_\infty \leq C n$; 
		\item $\| \bar{\nabla}\bar{\nabla} \chi_n \|_\infty \leq C n^2$,
	\end{enumerate}
	where $\bar{\nabla}$ denotes the covariant derivative taken with respect to $\overline{g}$. 
	In particular, we have that $\| X_i \chi_n \|_\infty \leq C n$ and $\|X_i^* \chi_n \| \leq C n$ and $\| \Delta \chi_n \|_\infty \leq C n^2$ for all $i = 1,2,\dots,m$ and all $n \in \mathbb{N}.$
\end{lemma}

\begin{proof}[Sketch of the proof]
	We follow the strategy of \cite{shi1989deforming}, taking a function $\eta \in C^\infty (\mathbb{R})$ such that $\eta(x) = 1$ when $x \leq 0$, $\eta(x) = 0$ when $x \geq 1$ and $\eta(x) \in [0,1]$. Moreover, we require that 
	\[
	\begin{cases}
		\eta'(x) \leq 0, \qquad &\forall \; x \in \mathbb{R} \\
		|\eta''(x)|\leq 8, \qquad &\forall \; x \in \mathbb{R} \\ 
		\frac{|\eta'(x)|^2}{\eta(x)} \leq 16, \qquad &\forall \; x \leq 1.
	\end{cases} 
	\]
	Let $\delta > 0$ and put 
	\[
	\xi_\delta (x) := \eta \left( \frac{d(x,\Gamma) - \delta/2}{\delta/4} \right), \qquad \forall x \in \Omega.
	\]
	With the same computations as \cite{shi1989deforming}, we have $\xi_\delta \in C^\infty (\Omega)$, $\xi_\delta(x) = 1$ when $d(x,\Gamma) \leq \delta/2$, $\xi_\delta(x) = 0$ when $d(x,\Gamma) \geq 3\delta/4$ and $\xi_\delta(x) \in [0,1]$. Moreover, 
	\[
	\bar{\nabla}_\beta \xi_\delta (x) = \frac{4}{\delta} \eta' \cdot \bar{\nabla}_\beta d(x,\Gamma)
	\]
	and 
	\[
	\bar{\nabla}_\alpha \bar{\nabla}_\beta \xi_\delta(x) = \frac{4}{\delta} \eta' \cdot \bar{\nabla}_\alpha \bar{\nabla}_\beta d(x,\Gamma) + \frac{16}{\delta^2} \eta'' \cdot \bar{\nabla}_\alpha d(x,\Gamma) \cdot \bar{\nabla}_\beta d(x,\Gamma).
	\]
	Since $d(x,\Gamma)$ is a distance function, we have $\| \bar{\nabla} d(x,\Gamma) \| \leq 1$ so the gradient of $\xi_\delta$ is bounded by $C/\delta$. The first term in the Hessian is bounded by a version of the Hessian comparison theorem for the distance function $d(x,\Gamma)$, see for example \cite[Theorem 1.6]{chahine2020volume}. The required lower bound on the $k$-Ricci curvature is obtained by letting $\bar{g}$ be the Euclidean metric if $M = \mathbb{R}^n$ and $\bar{g}$ arbitrary if $M$ is a closed manifold. In the latter case, the curvature assumption follows from compactness. The result is finally obtained by choosing $\delta_n \sim \frac{1}{n}$. 
\end{proof}

We now prove Courant's theorem for the most general case of mixed boundary conditions. Let $\Gamma \subseteq \partial \Omega$ denote the part of the boundary where we enforce Dirichlet boundary conditions. Let $u$ be an eigenfunction of $- \Delta_\bullet$ and let $\Omega_i$ be a nodal domain. We consider two cases, as in Figure \ref{fig:domain}: 
\begin{enumerate}
	\item $\overline{\Omega_i} \cap \partial \Omega = \varnothing$;
	\item $\overline{\Omega_i} \cap \partial \Omega \neq \varnothing$.
\end{enumerate}

\begin{figure}
    \centering
    \includegraphics[scale = 0.5]{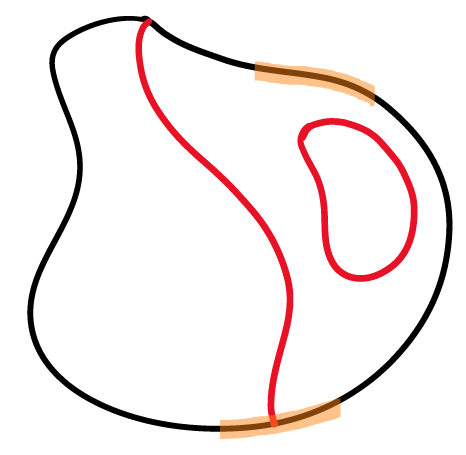}
    \caption{A domain $\Omega \subseteq M$ with nodal set $Z_u$ in red and Dirichlet boundary conditions enforced on $\Gamma \subseteq \partial \Omega$ in orange. One of the nodal domains is separated from the boundary of $\Omega$, while the other two touch it.}
    \label{fig:domain}
\end{figure}

In the first case, the boundary of $\Omega_i$ is a subset of $Z_u$, so that $u|_{\Omega_i}$ solves a spectral problem with Dirichlet boundary conditions. In this case, $u|_{\Omega_i} \in S_0^1(\Omega_i)$. This case is treaded in detail in \cite{ES}, hence we focus on the second case, where instead $u|_{\Omega_i}$ solves a mixed spectral problem on $\Omega_i$. Dirichlet boundary conditions are enforced on $\Gamma_i := \overline{\Omega_i} \cap (\Gamma \cup Z_u)$, and we will show that $u|_{\Omega_i} \in S^1_{0,\Gamma_i}(\Omega_i)$.

\begin{lemma}\label{lem:cutoff}
Suppose that \textbf{(S)} holds. Consider an eigenfunction $u$ for $- \Delta_\bullet$ corresponding to $\lambda_k$ and a nodal domain $\Omega_i$. Then, the restriction $u|_{\Omega_i}$ belongs to the space $S_{0,\Gamma_i}^1(\Omega_i)$, where $\Gamma_i$ is defined above. Moreover, $u|_{\Omega_i} \in D(- \Delta_{Z, \Gamma_i})$ and it is an eigenfunction corresponding to $\lambda_k$. Furthermore, if $\psi_i$ is defined by \begin{equation}\label{eq:psii}
	\psi_i(x) := \begin{cases}
		u(x), \qquad &\forall \; x \in \Omega_i \\
		0, \qquad   &\text{otherwise}
	\end{cases}
\end{equation}
then $\psi_i \in S^1_\bullet(\Omega).$
\end{lemma}

\begin{proof}
Throughout this proof, we fix $i$ and consider the nodal domain $\Omega_i$. Denote by $v$ the restriction $u|_{\Omega_i}$. Note that, by \textbf{(S)}, the function $v$ is smooth up to the boundary $\partial \Omega_i$, as it is the restriction of $u$, which is smooth up to $\partial \Omega$. The first step is to prove that $v \in S^1_{0,\Gamma_i}(\Omega_i)$. \\

 Take cutoff functions $\chi_n$ as in Lemma \ref{lem:cutoff_existence} applied to $\Omega_i$, and define $u_n = \chi_n \cdot v$. It is clear that $u_n \in C^\infty_{0,\Gamma_i}(\Omega_i)$. The same estimates as in \cite[Lemma 2.0.1]{anne1992bornes} show that $u_n \xrightarrow{n \to \infty} v$ in $S^1$-norm. For example, note that
 \[
 \|X_j(u_n - v)\| = \|X_j ((1 - \chi_n) v) \| \leq \|X_j(1 - \chi_n) \cdot v \| + \|(1 - \chi_n) \cdot  X_j v \|
 \]
 for all $j$. The second term tends to zero because $X_j v \in L^2$. The function $X_j(1 - \chi_n)$ has support only close to $\Gamma_i$, where $v$ vanishes. Thus, the mean value theorem can be used to show that the first term also goes to zero as $n \to \infty.$ For the details, see \cite[page 5]{anne1992bornes}\footnote{For the convenience of the reader, the mean value theorem is called \emph{le th\'eor\`eme des accroissements finis} in French.}. Thus, $\nabla_H u_n \to \nabla_H v$ in $L^2$ and $v \in S^1_{0,\Gamma_i}(\Omega_i)$. \\

	 To prove that $v \in D(- \Delta_{Z, \Gamma})$, we have to prove that $w \mapsto q^{Z, \Gamma_i}_{\Omega_i}(v,w)$ extends to a bounded linear map on $L^2$. First, let $w \in C^\infty_{0,\Gamma_i}(\Omega_i)$. Recalling that $u_n \to v$, by definition, we have
 \begin{align*}
	 	q_{\Omega_i} (v, w) &= \lim\limits_{n \to \infty} q_{\Omega_i} (\chi_n \cdot v,w)\\
	 	 &= \lim\limits_{n \to \infty} \int_{\Omega_i} g(\nabla_H(\chi_n \cdot  v), \nabla_H w) \; \omega \\
	 	 &= \lim\limits_{n \to \infty} \int_{\Omega_i} - \Delta u_n \cdot  w \; \omega \text{ + \cancel{boundary term}} \\
	 	 &= \lim\limits_{n \to \infty} \langle \Delta u_n , w \rangle =  \langle \Delta v , w \rangle = - \lambda_k \langle v , w \rangle,
	 	\end{align*}
	where the boundary term is zero because of Proposition \ref{prop:boundary}. By density, the result extends to $S^1_{0,\Gamma_i}(\Omega_i)$. We conclude that $w \mapsto q^{Z, \Gamma_i}_{\Omega_i}(v,w)$ is bounded and that $u$ is an eigenfunction corresponding to $\lambda_k$. \\
	 
	 Finally, we show that the function $\psi_i \in S^1_{0,\Gamma}(\Omega)$ so that it is a valid test function in the min-max principle. Define the function $\psi_{n,i}$ by 
	 \[
	 \psi_{n,i}(x) := \begin{cases}
	 	u_n(x), \qquad & \forall \; x \in \Omega_i \\
	 	0, \qquad &\text{otherwise}
	 \end{cases}
	 \]
	It is clear that $\psi_{n,i} \in C^\infty_{0,\Gamma}(\Omega)$ and that $\psi_{n,i} \to \psi_i$ tends to $\psi_i$ in $S^1$-norm, so the proof is complete. 
\end{proof}

\begin{proof}[Proof of Courant's theorem]
    Consider the case of mixed boundary conditions. Let $u$ be an eigenfunction associated to $\lambda_k = \lambda_k^{Z,\Gamma}(\Omega)$, where $\Gamma \subseteq \partial \Omega$ has finitely many connected components. Assume for contradiction that $u$ has at least $k + 1$ nodal domains $\Omega_1 , \dots , \Omega_{k + 1}$.

For a fixed nodal domain $\Omega_i$, define the function $\psi_i(x)$ as in (\ref{eq:psii}). 
The restriction $\psi_i|_{\Omega_i} = u|_{\Omega_i}$ lies in $S^1_{0,\Gamma_i} (\Omega_i)$. Furthermore, it belongs to the domain $D(- \Delta_{Z,\Gamma_i})$ of the sub-Laplacian on $\Omega_i$ and it is an eigenfunction corresponding to $\lambda_k$. By the min-max principle, the Rayleigh quotient of $\psi_i$ is $\lambda_k$. Indeed: 
\begin{equation}
	R^{Z,\Gamma}_\Omega [\psi_i] = \frac{\int_{\Omega} g(\nabla_H \psi_i , \nabla_H \psi_i) \; \omega}{\int_\Omega |\psi_i|^2 \; \omega} = \frac{\int_{\Omega_i} g(\nabla_H \psi_i , \nabla_H \psi_i) \; \omega}{\int_{\Omega_i} |u|^2 \; \omega} = R^{Z,\Gamma_i}_{\Omega_i}\left[ u|_{\Omega_i} \right] = \lambda_k.
\end{equation}
Consider the $k$-dimensional subspace $L := \spn \{ \psi_1 , \dots , \psi_k \}$. Note that $L \subseteq H^1_{0,\Gamma}(\Omega)$ and that $R^{Z,\Gamma}_{\Omega}[f] = \lambda_k$ for all $f \in L$. If the orthonormal basis of eigenfunctions is denoted by $\{ u_j : j \in \mathbb{N} \}$, note that we can choose $f \in L$ such that $f$ is orthogonal to $u_1,u_2,\dots,u_{k - 1}$. Again by the min-max principle, we have that 
\begin{equation}
	\lambda_k = \min_{v \in S_0^1(\Omega) \setminus \{0\}, \; v \perp u_1 , u_2 , \dots u_{k - 1} } R_\Omega[v],
\end{equation}
and this minimum is attained if and only if $v$ is an eigenfunction corresponding to $\lambda_k$. As such, the function $f$ chosen above is an eigenfunction corresponding to $\lambda_k$ that vanishes identically on the nodal domain $\Omega_{k + 1}.$ Thus, by unique continuation, $f = 0$ identically on $\Omega$, a contradiction.

\end{proof}

\subsection{Second proof}\label{sec:second}

In this subsection, we present a second proof of Lemma \ref{lem:cutoff}, and hence a second proof of Courant's theorem. This proof is based on the recent paper \cite{frank2024courant}, and avoids assumption \textbf{(S)}. As we have seen, the main difficulty in the proof is to show that the restriction of an eigenfunction to a nodal domain lives in the correct Sobolev space. We achieve this by means of the following result, which is a slight generalization of \cite[Theorem 2.2]{frank2024courant}.

\begin{theorem}[Improved restriction]\label{thm:ImprovedRestriction}
	Let $u \in S^1_{0,\Gamma}(\Omega) \cap C(\Omega)$ and let $\Omega_i$ be a connected component of $\{ u \neq 0 \}$. Then, $u|_{\Omega_i} \in S^1_{0,\Gamma_i}(\Omega_i)$, where $\Gamma_i = \overline{\Omega_i} \cap (\Gamma \cup Z_u)$.   
\end{theorem} 

\begin{proof}
    The first step is to prove that $\psi_i$ defined by equation (\ref{eq:psii}) lies in $S^1_{\mathrm{loc}}(\Omega)$. Let $\varphi \in C^\infty_0(\Omega)$ be a test function and note that $\psi_i \in C(\Omega)$. Then, $\varphi \cdot \psi_i \in C(\Omega)$ and $\varphi \cdot \psi_i$ agrees with $\varphi \cdot u$ on $\Omega_i$. The latter function lies in $S^1(\Omega_i) \cap C(\Omega_i)$, because $\varphi \in C^1(\Omega)$ and $u \in S^1(\Omega)$. Moreover, it vanishes on $\partial \Omega$, since $\varphi$ has compact support in $\Omega$. Hence, by the Vanishing Lemma, $(\varphi \cdot \psi_i) |_{\Omega_i} \in S^1_0(\Omega_i)$ (and hence also in $S^1_{0,\Gamma_i}(\Omega_i)$). As the test function $\varphi$ was arbitrary, we conclude that $\psi_i \in S^1_{\mathrm{loc}}(\Omega)$. \\

    The next step is to construct a bounded sequence $(v_j) \subseteq S^1_{0,\Gamma_i}(\Omega_i)$ which converges to $u|_{\Omega_i}$ in $L^2(\Omega_i)$. Assuming that this has been done, we conclude as follows: Extract a weakly convergent subsequence $v_j \xrightarrow{w} v$ in $S^1_{0,\Gamma_i}(\Omega_i)$. This subsequence converges weakly to $v$ in $L^2$ and strongly to $u|_{\Omega_i}$ in $L^2$. We conclude that $v = u|_{\Omega_i}$ almost everywhere, and thus $u|_{\Omega_i} \in S^1_{0,\Gamma_i}(\Omega_i)$. \\

    It remains to construct $v_j$. By definition, for any $u \in S^1_{0,\Gamma}$, there are functions $\varphi_j \in C^\infty_{0,\Gamma}(\Omega)$ such that $\varphi_j \to u$ in $S^1$-norm, where we recall that
	\[
	C^\infty_{0,\Gamma}(\Omega) = \{ \varphi \in C^\infty(\Omega) : \supp(\varphi) \cap \overline{\Gamma} = \varnothing \}.
	\] 
	We then define 
	\begin{equation}
		w_j = \min (\psi_i, (\varphi_j)_+)
	\end{equation}
	and $v_j = w_j|_{\Omega_i}$. By the first part of the proof, $w_j \in S^1_{\mathrm{loc}}(\Omega)$.The same estimates as \cite[p. 12]{frank2024courant} show that
	\[
	\sum_{k = 1}^m \int_{\Omega_i} |X_k v_j|^2 \; \omega \leq \sum_{k = 1}^m \int_\Omega \left( |X_k u|^2 + |X_i \varphi_j|^2 \right) \; \omega 
	\]
	and 
	\[
	\int_{\Omega_i} (v_j - u)^2 \; \omega \leq \int_\Omega (u - \varphi_j)^2 \; \omega \to 0, \qquad j \to \infty 
	\]
	These bounds show - since $u \in S^1(\Omega)$ and $(\varphi_j)$ converges in $S^1(\Omega)$ - that $v_j \in S^1(\Omega)$ and $(v_j)$ is a bounded sequence in $S^1$ and it converges to $u|_{\Omega_i}$ in $L^2$. It remains to prove that $v_j \in S^1_{0,\Gamma_i}(\Omega_i)$. \\
	
	This follows from an application of the vanishing Lemma. Note that $w_j \in C(\Omega)$ and that its restriction $v_j$ to $\Omega_i$ vanishes on $\Gamma_i$. Indeed, $\Gamma_i$ consists of parts coming from the nodal line $Z_u$ and parts coming from $\Gamma$. The function $w_j$ vanishes on both parts: On $Z_u$ because $\psi_i$ vanishes there and on $\Gamma$ because the support of $\varphi_j$ is disjoint from $\overline{\Gamma}$. This completes the proof. 
\end{proof}

Note that Theorem \ref{thm:ImprovedRestriction} easily implies Lemma \ref{lem:cutoff}, so that we have proven Courant's theorem independently of assumption \textbf{(S)}.

\section{Dirichlet-Cheeger inequality}\label{sec:DC}

Let $M$ be a CC-space and let $\Omega \subseteq M$ be a bounded open connected subset with piecewise smooth boundary. To prove a Cheeger inequality, we follow the steps outlined in \cite{TSG}. 

\begin{lemma}[Layer cake]
	Let $M$ be a smooth manifold with smooth volume $\omega$. If $f \geq 0$ is a smooth function, then
	\begin{equation}
		\int_0^\infty \omega( \{f > t \} ) \dd{t} = \int_M f \; \omega,
	\end{equation}
    where $\omega( \{f > t \} ) := \omega ( \{ p \in M : f(p) > t \} )$ is the $\omega$-volume of the super-level set. 
\end{lemma}

\begin{proof}
	This follows directly from \cite[Theorem 1.13]{LiebLoss}. 
\end{proof}

\begin{definition}[Dirichlet-Cheeger constant]
	Let $M$ be a CC-space and $\Omega \subseteq M$ a bounded domain with piecewise smooth boundary. Define the Dirichlet-Cheeger constant by 

    \begin{equation}\label{eq:DC-constant}
		h_D (\Omega) := \inf_A \frac{P_H(A;\Omega)}{\omega(A)},
	\end{equation}
	where the infimum is taken over all bounded subsets $A \subseteq \Omega$ with piecewise smooth boundary $\partial A$, which are compactly contained in $\Omega$.

\end{definition}

\begin{remark}
    In view of (\ref{eq:surfacemeasure-perimeter}), we could replace the numerator by $\sigma(\partial A)$, where $\sigma (\partial A) = \left| \oint_{\partial A} \iota_{\vb{n}_H} \omega \right|$ and $\vb{n}_H$ is a horizontal normal to the boundary. We take this point of view in the next section.
\end{remark}

\begin{lemma}\label{lem:CheegerLemma}
	If $f \geq 0$ is a smooth function which vanishes on $\partial \Omega$, then
	\begin{equation}
		\int_\Omega |\nabla_H f| \; \omega \geq h_D(\Omega) \int_\Omega f \; \omega.
	\end{equation}
\end{lemma}

\begin{proof}
	By the coarea formula (\ref{eq:coarea}), we have
	\[
	\int_\Omega |\nabla_H f| \; \omega = \int_0^\infty P_H (\{ f > t \}; \Omega) \dd{t}.
	\]
	Because $f$ vanishes on $\partial \Omega$, the set $\{ f > t \}$ is compactly embedded in $\Omega$, so it can be used as a ``test set'' for the Dirichlet-Cheeger constant.

    We conclude that 
	\[
	h_D(\Omega) \leq \frac{P_H(\{ f > t \};\Omega)}{\omega(\{ f > t \})},
	\]
	and thus 
	\[
	\int_\Omega |\nabla_H f| \; \omega \geq h_D (\Omega )  \int_0^\infty \omega(\{f > t\}) \dd{t} =  h_D(\Omega) \int_\Omega f \; \omega,
	\]
	by the layer cake representation. 
\end{proof}

\begin{theorem}[Dirichlet-Cheeger inequality]
	Let $M$ be a CC-space, and let $\Omega \subseteq M$ a bounded domain with piecewise smooth boundary. Then,
	\begin{equation}\label{eq:DC}
	\lambda_1^D(\Omega) \geq \frac{1}{4} h_D(\Omega)^2.
	\end{equation}
\end{theorem}

\begin{proof}
	Let $u$ be an eigenfunction corresponding to $\lambda_1^D(\Omega)$. Note that 
	\[
	\int_\Omega |\nabla_H(u^2)|\;\omega = 2\int_{\Omega} |u| |\nabla_H u| \; \omega \leq 2 \|u\| \|\nabla_H u\| = 2 \sqrt{\lambda_1^D(\Omega)} \|u\|^2,
	\]
	by the min-max principle. Next, applying Lemma \ref{lem:CheegerLemma} to $u^2$, we obtain
	\[
	\int_\Omega |\nabla_H (u^2)| \; \omega \geq h_D(\Omega) \int_{\Omega} u^2 \; \omega =  h_D(\Omega) \|u\|^2.
	\]
	Combining the previous two lines and rearranging gives (\ref{eq:DC}). 
\end{proof}

\begin{remark}
    To prove that the bound (\ref{eq:DC}) is non-trivial, one would need an isoperimetric inequality on CC-spaces of the form $\frac{\sigma (\partial A)}{\omega(A)} \geq C$. Such inequalities are obtained for CC-spaces where $\dim (\mathcal{D}_p) = k$ is constant in \cite{prandi2019sub}. In particular, when $M$ is a Carnot group (see Section \ref{sec:carnot}), we have \cite[Corollary 9]{prandi2019sub} that 
    \[
    \frac{\sigma(\partial A)}{\omega (A)} \geq \frac{2 \pi |\mathbb{S}^{k - 1}|}{|\mathbb{S}^k| \mathrm{diam}_H(A)},
    \]
    where $k = \dim(\mathcal{D}_p)$, which is always constant for Carnot groups, and $\mathrm{diam}_H(A)$ is the \emph{horizontal diameter} of $A$. 
\end{remark}

\section{Neumann-Cheeger inequality}\label{sec:NC}

A similar inequality holds for Neumann boundary conditions. The proof is slightly more technical than in the Dirichlet case. For instance, it relies on the nodal domain Theorem, and thus on assumption \textbf{(C)}. 
The structure of the proof is based on \cite{TSG}. We first need to define a measure $\sigma$ on the hypersurfaces of $M$ and a Cheeger constant that is adapted to the Neumann boundary conditions. 

\begin{definition}[Surface measure]
   Let $\Sigma \subseteq M$ be a smooth hypersurface. Define 
   \begin{equation}\label{eq:surfacemeasure}
       \sigma (\Sigma) :=  \int_{\Sigma} \iota_{\vb{n}_H} \omega,
   \end{equation}
   where $\vb{n}_H$ is a smooth horizontal normal to $\Sigma$ chosen such that the integral in (\ref{eq:surfacemeasure}) is positive.
\end{definition}

\begin{remark}
    For the Dirichlet-Cheeger inequality, we effectively defined a surface measure $P_H(A; \Omega)$ only for hypersurfaces which are the boundary of some open set $A \subseteq \Omega$. In the Neumann case, we need a notion of area for more general hypersurfaces $\Sigma \subseteq \Omega$. Luckily, the definition (\ref{eq:surfacemeasure}) agrees with the horizontal perimeter for boundaries $\partial A$ of open subsets $A \subseteq \Omega$ because of (\ref{eq:surfacemeasure-perimeter}). 
\end{remark}

\begin{definition}[Neumann-Cheeger constant]
	Let $\Omega \subseteq M$ be a bounded connected domain with piecewise smooth boundary. We define the Neumann-Cheeger constant by
	\begin{equation}
	h_N(\Omega) := \inf_\Sigma \frac{\sigma(\Sigma)}{\min \{ \omega(\Omega_1) , \omega(\Omega_2) \}},
	\end{equation} 
	where the infimum is taken over all piecewise smooth (not necessarily connected) hypersurfaces $\Sigma \subseteq \Omega$ that separate $\Omega$ into two disjoint open sets $\Omega_1$ and $\Omega_2$. 
\end{definition}

The idea for proving Cheeger's inequality in the Neumann case is to use the nodal set $\{ u = 0 \}$ as a separating hypersurface. To make this rigorous, we first require a lemma.

\begin{lemma}\label{lem:nodalneumann}
	Let $M$ be a CC-space, and let $\Omega \subseteq M$ be a connected bounded domain with piecewise smooth boundary. Assume that \textbf{(C)} holds. Then, the first eigenvalue $\lambda^N_1(\Omega) = 0$, and its corresponding eigenfunctions are constant. Moreover, $\lambda_2^N(\Omega) > 0$ and its corresponding eigenfunctions have precisely two nodal domains. 
\end{lemma}

\begin{proof}
	It is clear that any constant function is an eigenfunction corresponding to $\lambda = 0$. Conversely, assume that $u$ is an eigenfunction corresponding to $\lambda = 0$. Then, by the min-max principle, $R^N_\Omega[u] = 0$, and hence $X_i u = 0$ for all $i = 1,2,\dots,m$. Since any two points in $\Omega$ can be connected by a horizontal curve, it follows that $u$ is constant. Thus, $\lambda^N_2(\Omega) > 0$. By Courant's nodal domain theorem, any corresponding eigenfunction has at most 2 nodal domains. To see that it has \textit{at least} 2 nodal domains, we argue by contradiction. If $u$ would have only one nodal domain, it would not be allowed to change sign in $\Omega$. However, by the weak spectral theorem, $u$ is orthogonal to constants, which is a contradiction. 	
\end{proof}

The Neumann-Cheeger inequality is a lower bound on $\lambda_2^N(\Omega)$. Let $u$ be a corresponding eigenfunction with nodal set $Z_u$ and nodal domains $\Omega_\pm$. Without loss of generality, assume that $\omega(\Omega_+) \leq \omega (\Omega_-)$. If not, replace $u$ by $-u$.  We consider the restriction $u|_{\Omega_+},$ where $\Omega_+ = \{ x \in \Omega : u(x) > 0 \}.$ The boundary of $\Omega_+$ consists of $\partial \Omega$ and $Z_u$, and $u|_{\Omega_+}$ solves a mixed spectral problem. Indeed: 

\begin{lemma}\label{lem:mixed}
	Let $M$ be a CC-space, and let $\Omega \subseteq M$ be a connected bounded domain with piecewise smooth boundary. Then, in the setting above, $u|_{\Omega_+}$ is an eigenfunction of the mixed spectral problem on $\Omega_+$ with Dirichlet boundary conditions enforced on $Z_u$. Moreover, the corresponding eigenvalue is $\lambda_1^Z(\Omega_+,Z_u)$. 
\end{lemma}

\begin{proof}

It follows from Lemma \ref{lem:cutoff} that $u_+ := u|_{\Omega_+}$ is an eigenfunction of the mixed spectral problem, so it remains to show that it corresponds to the \emph{first} eigenvalue. To this end, we prove that there exists an eigenfunction $u_1$ corresponding to $\lambda_1^Z(\Omega, Z_u)$ that is strictly positive in $\Omega_+$. Assuming we have shown this, let $\lambda$ denote the eigenvalue corresponding to $u_+.$ If $\lambda > \lambda_1^Z(\Omega_+, Z_u)$, then $u_+$ would have to be orthogonal to $u_1$, which is not possible, as both functions are strictly positive in $\Omega_+$. \\

Thus, we only need to prove that there is a positive eigenfunction. Let $u_1$ be a nontrivial eigenfunction corresponding to $\lambda_1^Z(\Omega_+, Z_u)$. As in \cite{ES}, we note that $|u_1| \in D(\hat{q}^{Z,Z_u}_{\Omega_+})$ and that $R^{Z,Z_u}_{\Omega_+}[|u_1|] = R^{Z,Z_u}_{\Omega_+}[u_1]$. This follows essentially from the result of \cite[Exercise 5.18]{Evans}. Hence, $|u_1|$ is also a minimizer of the Rayleigh quotient and hence it is an eigenfunction corresponding to $\lambda_1^Z(\Omega_+ , Z_u)$. In conclusion, $|u_1|$ is the desired non-negative eigenfunction and the proof is complete.
\end{proof}

We now establish a Cheeger inequality for mixed spectral problems. 

\begin{definition}[Mixed Cheeger constant]
Let $M$ be a CC-space, and let $\Omega \subseteq M$ a bounded domain with piecewise smooth boundary. Let $\Gamma \subseteq \partial \Omega$ with finitely many connected components. We define the mixed Cheeger constant by 
\begin{equation}
h_Z(\Omega, \Gamma) := \inf_A \frac{P_H(A; \Omega)}{\omega (A)} = \inf_A \frac{\sigma (\partial A)}{\omega (A)},
\end{equation}	
where the infimum is taken over all open sets $A \subseteq \Omega$ with piecewise smooth boundary such that $\partial A \cap \Gamma = \varnothing.$ Since the proof is practically identical to that of the Dirichlet-Cheeger inequality, we omit the proof. 
\end{definition}

\begin{theorem}[Mixed Cheeger inequality]\label{thm:mixedcheeger}
Let $M$ be a CC-space, and let $\Omega \subseteq M$ a bounded domain with piecewise smooth boundary. Let $\Gamma \subseteq \partial \Omega$ with finitely many connected components. Then, we have 
\begin{equation}
\lambda_1^Z(\Omega, \Gamma) \geq \frac{1}{4} h_Z(\Omega, \Gamma)^2.
\end{equation} 
\end{theorem}

\begin{lemma}\label{lem:comparisonDN}
	Let $M$ be a CC-space, and let $\Omega \subseteq M$ a bounded domain with piecewise smooth boundary. Let $\Gamma \subseteq \partial \Omega$ with finitely many connected components. Then, we have
	\begin{equation}
		h_{Z}(\Omega_+,Z_u) \geq h_N(\Omega). 
	\end{equation}
\end{lemma}

\begin{proof}
Assume for contradiction that $h_{Z}(\Omega_+,Z_u) < h_N(\Omega)$. By definition of infimum, there must be a specific set $A \subseteq \Omega_+$ with piecewise smooth boundary and $A \cap Z_u  = \varnothing$ such that 
\[
\frac{\sigma(\partial A \cap \Omega_+)}{\omega(A)} < h_N(\Omega).
\]
Notice that $\Gamma := \partial A \cap \Omega_+$ is a hypersurface meeting the definition of $h_N(\Omega)$, so that 
\[
h_N(\Omega) \leq \frac{\sigma(\Gamma)}{\min \{  \omega(\Omega_+), \omega(\Omega_-) \}} = \frac{\sigma(\Gamma)}{\omega(\Omega_+)}\leq \frac{\sigma(\Gamma)}{\omega(A)} < h_N(\Omega),
\]
where we used that $\omega(\Omega_-) \geq \omega(\Omega_+) \geq \omega(A)$. This is a contradiction, so the proof is complete. 
\end{proof}

\begin{theorem}[Neumann-Cheeger inequality]
	Let $M$ be a CC-space, and let $\Omega \subseteq M$ be a connected bounded domain with piecewise smooth boundary. Assume that \textbf{(C)} holds. Then, we have 
	\begin{equation}
	\lambda_2^N(\Omega) \geq \frac{1}{4}h_N(\Omega)^2
	\end{equation}
\end{theorem}

\begin{proof}
	It follows from successively applying Lemma \ref{lem:mixed}, Theorem \ref{thm:mixedcheeger} and Lemma \ref{lem:comparisonDN} that 
	\[
	\lambda_2^N(\Omega) = \lambda_1^Z(\Omega_+, Z_u) \geq \frac{1}{4} h_Z^2(\Omega_+, Z_u) \geq \frac{1}{4} h^2_N(\Omega),
	\]
	completing the proof.
\end{proof}

\section{Max flow min cut result}\label{sec:mfmc}

In this section, we present a technique to lower bound the Cheeger constants, both in the Dirichlet and Neumann case. This approach is based on \cite{grieser2006first}.

\begin{theorem}[Max flow min cut]\label{thm:MFMC}
    Let $M$ be a CC-space, and let $\Omega \subseteq M$ a bounded domain with piecewise smooth boundary. Let $V \in \mathfrak{X}_H(\Omega)$ be a horizontal vector field and let $h \in \mathbb{R}$ be such that $\|V\| \leq 1$ and $\di_\omega (V) \geq h$ pointwise in $\Omega$. Then, $h_D(\Omega) \geq h$. 
\end{theorem}

\begin{proof}
    Let $V$ be as above, and let $A\subseteq \Omega$ be a test set for the Dirichlet-Cheeger constant. The basic chain of inequalities is as follows: 
    \[
    \sigma(\partial A) \geq \oint_{\partial A} \iota_V \omega = \int_A \di_\omega (V) \; \omega \geq h \cdot  \omega (A),
    \]
    The equality in the middle is just the divergence theorem, while the last inequality is trivial, so it remains to prove the first inequality. \\

    If $p \in \partial A$ is characteristic, then $V_p \in \mathcal{D}_p \subseteq T_p(\partial A)$. Hence if $\vb{v}_1 , \dots , \vb{v}_{n-1}$ is a basis for $T_p(\partial A)$, then
    \[
    \iota_V \omega_p (\vb{v_1 , \dots , \vb{v}_{n-1}}) = \omega_p (V, \vb{v_1 , \dots , \vb{v}_{n-1}}) = 0,
    \]
    because the latter vectors are linearly dependent. It thus suffices to consider non-characteristic points. Decompose $V = a \vb{n}_H + W$ with $a \in \mathbb{R}$ and $W \in \mathcal{D}_p \cap T_p (\partial A)$ and note that $a = g(V , \vb{n}_H)$. Now, denoting the set of non-characteristic points on $\partial A$ by $\{ \textbf{NC} \}$, we have 
    \[
    \oint_{\partial A} \iota_{V} \omega = \oint_{\partial A \cap \{ \text{NC} \}  } g(V, \vb{n}_H) \; \iota_{\vb{n}_H} \omega ,
    \]
    and $g(V, \vb{n}_H) = \|V\| \|\vb{n}_H \| \cos \vartheta \leq 1$. Thus,
    \[
    \oint_{\partial A} \iota_V \omega \leq \oint_{\partial A \cap \{ \text{NC} \}  } \iota_{\vb{n}_H} \omega = \sigma(\partial A).
    \]
    The conclusion is that $\frac{\sigma(\partial A)}{\omega (A)} \geq h$, but the test set $A$ is arbitrary, so the same holds for the Cheeger constant. 
\end{proof}

We now extend this proof to the Neumann case. Let $\Sigma$ be a hypersurface cutting $\Omega$ into two connected components $\Omega_1$ and $\Omega_2$. Without loss of generality, let $\Omega_1$ have smallest volume among them. Then, the boundary of $\Omega_1$ consists of parts coming from $\partial \Omega$ and parts coming from $\Sigma$. In fact, if we can assume that 
\begin{equation}\label{eq:inward}
    \int_{\partial \Omega \cap \partial \Omega_1} \iota_V  \omega \leq 0,
\end{equation}
then it follows that 
\[
\sigma (\Sigma) \geq \int_{\Sigma} \iota_V \omega \geq \int_{\Sigma} \iota_V \omega + \int_{\partial \Omega \cap \partial \Omega_1} \iota_V \omega = \int_{\partial \Omega_1} \iota_V \omega.
\]
Notice that equation (\ref{eq:inward}) always holds if we choose $V$ to be inward pointing along the boundary $\partial \Omega$.
By repeating the Dirichlet argument, we then have 
\[
\sigma(\Sigma) \geq \int_{\partial \Omega_1} \iota_V \omega \geq h \cdot \omega(\Omega_1) = h \cdot \min (\omega (\Omega_1), \omega (\Omega_2)). 
\]
As the separating hypersurface was arbitrary, we conclude that $h_N(\Omega) \geq h$. 
We have hence deduced:

\begin{theorem}[Min flow max cut, Neumann version]
     Let $M$ be a CC-space, and let $\Omega \subseteq M$ a bounded connected domain with piecewise smooth boundary. Let $V \in \mathfrak{X}_H(\overline{\Omega})$ be a horizontal vector field which is inward-pointing along $\partial \Omega$. Let $h \in \mathbb{R}$ such that $\|V\| \leq 1$ and $\di_\omega (V) \geq h$ pointwise in $\Omega$. Then, $h_N(\Omega) \geq h$. 
\end{theorem}

\section{Examples}\label{sec:examples}

In this last section, we discuss how our main result applies to Carnot groups and the (Baouendi-)Grushin structure on a cylinder. For Carnot groups, we relate our results to previous work on coarea formulas. We are able to replace the surface measure with the spherical Hausdorff measure, in the process getting a different geometric constant in the Cheeger inequality. In the case of Grushin, we compute explicitly the spectrum of a cylinder and give an upper bound on the corresponding Cheeger constant. 

\subsection{Carnot groups}\label{sec:carnot}

\begin{definition}
    A connected and simply connected Lie group $G$ is called a \emph{Carnot group} if its Lie algebra $\mathfrak{g}$ admits a stratification, i.e. a direct sum decomposition $\mathfrak{g} =  V_1 \oplus \dots \oplus V_s$ such that $[V_1,V_i] = V_{i + 1}$ for  $i = 1,2,\dots,s-1,$ and such that $[V_1,V_s] = 0$. The integer $s$ is called the \emph{step} of the stratification. 
\end{definition}

Choosing a basis $(\xi_1 , \dots, \xi_m)$ for $V_1 \subseteq \mathfrak{g}$, we obtain globally defined left-invariant vector fields $X_1 , \dots, X_m \in \mathfrak{X}(G)$. These vector fields equip $G$ with the structure of a CC-space, where the generating vector fields are $X_1 , \dots , X_m$. \\

A Carnot group may be identified with $\mathbb{R}^n$ equipped with a non-Abelian group operation. To see this, extend $(\xi_1 , \dots , \xi_m)$ to an adapted basis for $\mathfrak{g} = V_1 \oplus \dots \oplus V_s$, which we denote by $(\xi^1_{1}, \dots , \xi^1_{m_1}, \dots , \xi^s_{1}, \dots , \xi^s_{m_s})$. Denote the corresponding left-invariant vector fields by $X^i_{j}$ and note that $X^1_j = X_j$. Because $G$ is connected and simply connected, the exponential map
\begin{equation}
	F : \mathbb{R}^{m_1} \times \dots \times \mathbb{R}^{m_s} \to G, \qquad  F(x) := \exp ( \sum_{i = 1}^s \sum_{j_i = 1}^{m_i} x^j_{j_i} \xi^j_{j_i} ),  
\end{equation}
with $x = (x^1 , \dots , x^s) \in \mathbb{R}^{m_1} \times \dots \times \mathbb{R}^{m_s}$ is a diffeomorphism. We refer to the map $F$ as a system of \textit{graded exponential coordinates}. Furthermore, we can use $F$ to identify the Lie group $G$ with $(\mathbb{R}^n , *)$, where the group operation $* : \mathbb{R}^n \times \mathbb{R}^n \to \mathbb{R}^n$ can be computed explicitly by using the Campbell-Baker-Hausdorff formula \cite{cassano2016some}. \\

Carnot groups come equipped with two natural families of maps: translations and dilations. For $x \in G$, we define the (left-)translation 
\begin{equation}
	\tau_x : G \to G, \qquad  \tau_x (g) = xg. 
\end{equation}
Dilations are first defined at the level of the Lie algebra $\mathfrak{g} = V_1 \oplus \dots \oplus V_s$. An element $\xi \in \mathfrak{g}$ can be written uniquely as $\xi = \sum_{i = 1}^s \xi_i$ with $\xi_i \in V_i$. We then define for $r > 0$ the dilation $\delta_r : \mathfrak{g} \to \mathfrak{g}$ by 
\begin{equation}
	\delta_r (\xi) := \sum_{i = 1}^s r^i \xi_i.
\end{equation}
It can be proven that $\delta_r$ is an automorphism of $\mathfrak{g}$ with inverse $\delta_{1/r}$. Moreover, there is a group homomorphism ${\delta}_r : G \to G$ defined by forcing the diagram 
\[
\begin{tikzcd}
	\mathfrak{g} \arrow[r, "\delta_r"] \arrow[d, "\exp"'] & \mathfrak{g} \arrow[d, "\exp"] \\
	G \arrow[r, "\delta_r"]                      & G                    
\end{tikzcd}
\]
to commute \cite{srgmt}. We again use the same symbol $\delta_r$ to define the dilation on the group level.  \\

Let $G$ be a Carnot group, and let $d : G \times G \to [0,\infty)$ be a continuous map that makes $G$ into a metric space. We say that $d$ is a \textit{homogeneous distance} on $G$ if it respects the translations and dilations, i.e. if 
\begin{enumerate}
	\item $d(x,y) = d(\tau_g x, \tau_g y)$ for all $x,y,g \in G$;
	\item $d(\delta_r x, \delta_r y) = r d(x,y)$ for all $x,y \in G$ and all $r > 0$. 
\end{enumerate}
We note that the Carnot-Carath\'eodory distance $d_{CC}$ is always a homogeneous distance on $G$ \cite[Proposition 2.3.39]{srgmt}. In some situations, it may pay off to choose a different homogeneous distance on $G$. \\

Let $G$ be a Carnot group and fix an adapted basis $(\xi^1_{1}, \dots , \xi^1_{m_1}, \dots , \xi^s_{1}, \dots , \xi^s_{m_s})$ for $\mathfrak{g} = V_1 \oplus \dots \oplus V_s$. Taking graded exponential coordinates with respect to this basis identifies $G$ with $\mathbb{R}^{m_1} \times \dots \times \mathbb{R}^{m_s}$, where $m_j = \dim(V_j)$. We then have the following: 

\begin{theorem}[cf. \cite{franchi2003structure} Theorem 5.1]
	Let $G$ be a Carnot group and fix graded exponential coordinates as above. Then, there exist constants $\varepsilon_1 , \dots , \varepsilon_s \in (0,1]$ depending only on the group structure such that, defining 
	\begin{equation}
		d_\infty(p,0) = \max \left\{ \varepsilon_j \left( |p_j|_{\mathbb{R}^{m_j}} \right)^{1/j}, \; j = 1,2,\dots,s \right\},
	\end{equation}
	with $ p = (p_1 , \dots , p_s), \; p_j \in \mathbb{R}^{m_j}$ and 
	\begin{equation}
		d_\infty(p,q) = d_\infty(q^{-1} p , 0),
	\end{equation}
	gives a homogeneous distance on $G$. 
\end{theorem}

Using the distance function $d_\infty$ on $G$ allows for a generalized coarea formula, where the surface measure $\sigma$ is replaced by the spherical Hausdorff measure. 

\begin{definition}
    Let $G$ be a Carnot group with stratification $\mathfrak{g} = V_1 \oplus \dots \oplus V_s$. Define the \emph{homogeneous dimension} of $G$ as 
    \[
    Q = \sum_{j = 1}^s j \cdot \dim(V_j). 
    \]
    Define further for $a \geq 0$ and $t > 0$ the measures 
    \[
    \Phi^a_t (E) = \frac{\omega_a}{2^a} \inf \left\{  \sum_{i = 1}^\infty \mathrm{diam}(B_{r_i}(p_i)) : E \subseteq \bigcup_{i = 1}^\infty B_{r_i}(p_i), \; \mathrm{diam} (B_{r_i}(p_i)) \leq t\right \},
    \]
    where $E \subseteq G$ is measurable, $p_i \in G$, $B_{r_i}(p_i)$ are open $d_\infty$-balls, and $\omega_a = \frac{\pi^{a/2}}{\Gamma(1 + a/2)}$ is the volume of a Euclidean unit ball in dimension $a$. The \emph{$a$-dimensional spherical Hausdorff measure} is defined by 
    \[
    \mathscr{S}^{a}(E) = \lim\limits_{t \to \infty} \Phi^a_t(E). 
    \]
\end{definition}

With the above definitions, we have \cite{magnani2005coarea}:

\begin{theorem}[Generalized coarea formula]
    Let $G$ be a Carnot group of homogeneous dimension $Q$, let $A \subseteq G$ be a measurable set and let $u : A \to \mathbb{R}$ be Lipschitz continuous. Then, for any measurable function $h : A \to [0,\infty]$, we have 
    \begin{equation}
        \int_A h(x) |\nabla_H u(x)| \dd{x} = \alpha_{Q - 1} \int_{\mathbb{R}} \int_{u ^{-1} (s) \cap A} h(x) \dd{\mathscr{S}^{Q - 1}}(x) \dd{s}, 
    \end{equation}
    where $\mathscr{S}^{Q - 1}$ is the $(Q - 1)$-dimensional spherical Hausdorff measure, and $\alpha_{Q - 1}$ is a geometric constant depending only on the group structure. The measure $\dd{x}$ is the Lebesgue measure on $\mathbb{R}^n \simeq G$. 
\end{theorem}

It is possible to repeat the proof of the Cheeger inequality using the spherical Hausdorff measure as the surface measure, i.e. $\sigma(\Sigma) = \int \chi_{\Sigma}(x) \dd{\mathscr{S}^{Q - 1}}(x)$. Moreover, assumption \textbf{(C)} is automatically satisfied for Carnot groups, since the vector fields defining the CC-structure have polynomial coefficients \cite[Proposition 2.4]{cassano2016some}. Thus:

\begin{theorem}
    Let $G$ be a Carnot group and $\Omega \subseteq G$ a connected bounded domain with piecewise smooth boundary. Then, 
    \[
    \lambda_2^N(\Omega) \geq \frac{\alpha_{Q - 1}^2}{4} h_N^2(\Omega),
    \]
    where the Neumann-Cheeger constant is defined with respect to the surface measure $\sigma$ defined above. 
\end{theorem}

\begin{remark}
    In the Heisenberg group $\mathbb{H}^{2n + 1}$, the distance $d_\infty$ and the geometric constant can be calculated explicitly \cite[Corollary 6.5.5]{srgmt}: 
    \[
    \alpha_{Q - 1} = \frac{2 \omega_{2n - 1}}{\omega_{Q - 1}}. 
    \]
    In the simplest case of $\mathbb{H}^3$, the constant equals $\frac{3}{\pi}$, leading to 
    \[
    \lambda^N_2(\Omega) \geq \frac{1}{4} \left( \frac{3}{\pi} \right)^2 h_N^2(\Omega).
    \]
    The constant in the Cheeger inequality is thus slightly worse than in the Euclidean case. The homogeneous distance $d_\infty$ is explicitly given by using the bijection $$\mathbb{H}^{2n+1}\to \mathbb{C}^n \times \mathbb{R} :  (x,y,t) \mapsto (x_1 + i y_1,\dots,x_n + i y_n,t) .$$ In those coordinates, the group operation reads $$(z,t)(z',t') = (z + z', t + t' + 2 \Im \langle z,z'\rangle)$$
    and $d_\infty (p,q) := N(p ^{-1} q)$ where $p ^{-1} q = (z,t) \in \mathbb{C}^n \times \mathbb{R}$ and $N(z,t) = \max \{ |z|, |t|^{1/2} \}$. The metric $d_\infty$ is a natural distance function on the Heisenberg group. For example, it is used in finding a fundamental solution for the sub-Laplacian \cite{bramanti2014invitation}.  
\end{remark}

\subsection{Grushin structure}

On the manifold $M = \mathbb{R} \times \mathbb{S}^1$, consider the vector fields $X = \pdv{x}$ and $Y = x \pdv{y}$. These vector fields are the generating family of a CC-structure on $M$. We take the smooth volume form $\omega = \dd{x} \wedge \dd{y}$. We remark that this is not the Riemannian volume obtained by declaring $(X,Y)$ to be an orthonormal frame. That volume form would blow up at $x = 0$. With the smooth volume form $\omega$, the sub-Laplacian is sometimes called a Baouendi-Grushin type operator, see for example \cite{letrouit2023observability}. \\

The sub-Laplacian corresponding to $\omega$ is given by \[
\Delta u = \pdv[2]{u}{x} + x^2 \pdv[2]{u}{y} = - \lambda u. 
\]
Separating the variables $u(x,y) = v(x) w(y)$, we obtain $w(y) = e^{iny}$ with $n \in \mathbb{Z}$ and $v(x)$ satisfies the ODE 
\begin{equation}\label{eq:theODE}
    v''(x) + (\lambda - n^2 x^2) v(x) = 0. 
\end{equation}
Suppose we want to solve the Neumann problem on $(0,1) \times \mathbb{S}^1$, i.e. we impose that $v'(0) = v'(1) = 0$. When $n = 0$ we have the solutions $u(x,y) = v(x) = \cos(m \pi x)$ for $m = 0,1,2, \dots$, with corresponding eigenvalues $\lambda_{0,m} = m^2 \pi^2$.  For positive $n$, the change of variables
\begin{equation}
    x = 4^{-1/4} \frac{\xi}{\sqrt{n}},
\end{equation}
transforms the ODE (\ref{eq:theODE}) into 
\begin{equation}
    w''(\xi) - \left( \frac{1}{4} \xi^2 - \frac{\lambda}{2n} \right) w(\xi) = 0,
\end{equation}
which has linearly independent solutions \cite{abramowitz1988handbook}
\begin{equation}
    w_\pm(\xi) = D_{- \lambda/2n - 1/2}(\pm i \xi),
\end{equation}
where $D_\nu(z)$ denotes the parabolic cylinder function. Defining $w_1(\xi) = w_+(\xi) + w_-(\xi)$ and $w_2(\xi) = w_+(\xi) - w_- (\xi)$, we have 
\[
w_1'(0) = 0, \qquad w_2(0) = 0,
\]
so that enforcing the Neumann boundary conditions gives 
\begin{equation}
    w(\xi) = c w_1(\xi), \qquad w_1'(4^{1/4} \sqrt{n} ) = 0.
\end{equation}
The remaining equation $w'(4^{1/4} \sqrt{n} ) = 0$ can be solved numerically for $\lambda$. Approximate values for $\lambda_{n,m}$ are displayed in the table below. We see that the Neumann-Cheeger inequality in this case pertains to the eigenvalue $\lambda_{1,0} \approx 0.325$.

\begin{table}[ht]
\begin{tabular}{|l||l|l|l|}
\hline                                                 & $n$ = 0  & $n$ = 1  & $n$ = 2    \\ \hline \hline 
$m$ = 0                                            & 0      & 0.325  & 1.203    \\ \hline 
$m$ = 1                                            & 9.870  & 10.26  & 11.504   \\ \hline 
$m$  = 2 & 39.478 & 39.825 & 40.877 \\ \hline 
\end{tabular}
\end{table}

Let us consider the Neumann-Cheeger constant for the cylinder. There are two obvious ways of cutting the cylinder into two parts: Either along a circle $\{h\} \times \mathbb{S}^1$ with $h \in (0,1)$ or along two straight lines running up the cylinder $(0,1) \times \{\varphi_1,\varphi_2\}$ with $\varphi_1,\varphi_2 \in \mathbb{S}^1$. It is easily seen that the optimum among these two ways of cutting is obtained when $\varphi_1,\varphi_2$ are chosen to be diametrically opposite. In this case, one has
\[
\frac{\sigma((0,1) \times \{\varphi_1,\varphi_2\})}{\min (\omega(\Omega_1), \omega(\Omega_2))} =\frac{2}{\pi},
\]
so that $h_N(M) \leq \frac{2}{\pi} \approx 0.637$. Presuming that this is an equality, we would have 
\[
\lambda_{2}^N(M) = \lambda_{1,0} \approx 0.325 \geq \frac{1}{4} h_N(M)^2 \approx 0.10.
\]

Finally, let us give an example of application of Theorem \ref{thm:MFMC}. The vector field $V = x \pdv{x}$ satisfies $\di_\omega (V) = 1$ and $|V| = |x| \leq 1$. Thus,
\[ h_D(M) \geq 1. \]
By similar techniques as above, one can numerically determine the Dirichlet spectrum of $M$. The lowest eigenvalue corresponds to $n = 0$ with eigenfunction $u(x,y) = \sin(\pi x)$. This is clearly an eigenfunction, and it corresponds to $\lambda_1^D(M)$ because it is strictly positive in the interior. We see that 
\[
\pi^2 = \lambda_1^D(M) \geq \frac{1}{4} h_D^2(M) \geq \frac{1}{4}.
\]

Note that this vector field cannot be used to draw conclusions about $h_N$, as it is outward-pointing along the upper boundary of the cylinder. We postpone finding sharper bounds for the Cheeger constants to future research. \\

{\bf Acknowledgements.}  I would like to thank my PhD supervisor Marcello Seri for his fantastic guidance and support. 
Moreover, thanks to Dario Prandi and Romain Petrides for useful discussions during the development of this work and for their hospitality while I visited Paris. After its initial appearance on the arXiv, the work has been improved by discussions with Daniel Grieser, Rupert Frank, and Bernard Helffer.
This research is supported by NWO project OCENW.KLEIN.375.

\bibliographystyle{plain}
\bibliography{bibliography.bib}

\end{document}